\documentclass{amsart}
\usepackage{amssymb}

\usepackage{amsfonts,amsmath,amstext,amsbsy,euscript}
\copyrightinfo{2000}{American Mathematical Society}

\theoremstyle{plain}
\newtheorem{theorem}{Theorem}[section]
\newtheorem{proposition}[theorem]{Proposition}
\newtheorem{corollary}[theorem]{Corollary}
\newtheorem{lemma}[theorem]{Lemma}

\theoremstyle{definition}

\theoremstyle{remark}

\numberwithin{equation}{section}

\begin{document}

\title{Conical Distributions on the Space of Flat Horocycles}

\author{Fulton B. Gonzalez}
\address{Department of Mathematics,
Tufts University,
Medford, MA 02155}
\email{fulton.gonzalez@tufts.edu}
\subjclass[2000]{Primary: 43A85}

% Remove or comment out any unused author tags.
% author one information

% Use this \subjclass if you are using amsproc version 2.0 (December 1999).
%\subjclass[2000]{Primary: 43A77, Secondary: 43A90}
%\date{March 1, 2005}
%\keywords{Wave Equation, Spherical Harmonics}
% Use this one if you are using an older version of amsproc.
%\subjclass{}

\maketitle

\def\rar{\rightarrow} 
\def\fk{\mathfrak}
\def\e{\mathfrak e}
\def\g{\mathfrak g} 
\def\ge{\mathfrak g} 
\def\k{\mathfrak k} 
\def\a{\mathfrak a} 
\def\m{\mathfrak m} 
\def\n{\mathfrak n}
\def\p{\mathfrak p} 
\def\u{\mathfrak u} 
\def\t{\mathfrak t} 
\def\h{\mathfrak h} 
\def\z{\mathfrak z}
\def\d{\mathfrak d}
\def\q{\mathfrak q}
%-----------------------------------------------------------------------
% End of article.top
%-----------------------------------------------------------------------

\def\Bbb{\mathbb}
\def\Cal{\mathcal}
\begin{abstract}
Let $G_0=K\ltimes\mathfrak p$ be the Cartan motion group associated
with a noncompact semisimple Riemannian symmetric pair 
$(G,\,K)$.  Let $\frak a$
be a maximal abelian subspace of $\mathfrak p$ and let $\p=\a+\q$ be
the corresponding orthogonal decomposition.  A flat horocycle in $\p$ is a
$G_0$-translate of $\q$.   A conical distribution on the space $\Xi_0$
of flat horocycles is an eigendistribution of the algebra $\mathbb
D(\Xi_0)$ of $G_0$-invariant differential operators on $\Xi_0$ which
is invariant under the left action of the isotropy subgroup of $G_0$
fixing $\q$.   We prove that the space of conical distributions
belonging to each generic eigenspace of $\mathbb D(\Xi_0)$ is
one-dimensional, and we classify the set of all conical distributions on $\Xi_0$ 
when $G/K$ has rank one.
\end{abstract}
\section{Introduction and Preliminaries}
 In this paper we study the flat analogues of conical distributions on the
 space of horocycles associated with noncompact symmetric spaces. Let $G$ be a noncompact real semisimple Lie group with
finite center, let $\fk g$ be its Lie algebra, and let $K$ be a maximal compact
subgroup of $G$.   Let $\theta$ be the corresponding Cartan
involution of $G$, and we also let $\theta$ denote its differential on
$\mathfrak g$.  Let $\k$ be the Lie algebra of
$K$ and $\p$ its orthogonal complement relative to the Killing form
$B$ on $\g$, so that $\g$ has Cartan decomposition $\g=\k+\p$.  We will generally use the
notation in Helgason's books \cite{DS}, \cite{GGA}, and \cite{GASS}.  In particular, 
we let $\a$  denote a maximal abelian subspace of $\p$, $\Sigma$  the 
set of restricted roots of $\g$ relative to $\mathfrak a$,  $W$  the Weyl
group of $\Sigma$, $\ge_\alpha$  the restricted root space corresponding to
$\alpha\in\Sigma$ and $m_\alpha$ its dimension.
In addition, let $\mathfrak a^+$ denote a fixed Weyl chamber in $\mathfrak a$,
$\Sigma^+$ the corresponding positive system of restricted roots, and
$\rho=1/2\,\sum_{\alpha\in\Sigma^+} m_\alpha \alpha$.  We  put
$\n=\sum_{\alpha\in\Sigma^+}\ge_\alpha$, and let $N$ and $A$ be the 
analytic subgroups of $G$ with Lie algebras $\n$ and $\a$, respectively.
Then $G$ has Iwasawa decomposition $G=NAK$.  Finally, we let $M$ and $M'$
denote the centralizer and normalizer of $A$ in $K$, respectively.
Then $W=M'/M$.  We let $w$ be the order of $W$.

We identify $\mathfrak p$ with $\mathfrak p^*$ (respectively
$\mathfrak a$ with $\mathfrak a^*$) via the restriction of the Killing
form $B$ to $\mathfrak p$ (respectively $\mathfrak a$).  In this way,
elements of the (complexified) symmetric algebra $S(\mathfrak p)$ can
be viewed as polynomial functions on $\mathfrak p$, and also with
constant coefficient differential operators on $\mathfrak p$.   If
$p\in S(\mathfrak p)$, we let $\partial(p)$ be the corresponding
differential operator on $\mathfrak p$.

A {\it horocycle} in the symmetric space $X=G/K$ is 
an orbit of a conjugate of $N$ in $X$.  The group $G$ acts transitively on the space $\Xi$
of all horocycles, and the isotropy subgroup of $G$ fixing the
identity horocycle $\xi_0=N\cdot o$ is $MN$, so that $\Xi=G/MN$.  The
mapping $(kM,a)\mapsto ka\cdot\xi_0$ is a diffeomorphism of $K/M\times
A$ onto $\Xi$  (\cite{GASS}).

According to \cite{GASS}, Chapter II, the algebra $\mathbb D(\Xi)$ of $G$-invariant differential operators on $\Xi$ is
isomorphic to $S(\a)$, the symmetric algebra of $\a$, via
$$
D_p\,\varphi(k\exp H\cdot\xi_0)=\partial(p)_H\,\varphi(k\,\exp H\cdot \xi_0),\qquad (p\in S(\a))
$$
 If
$\mathfrak a^*_c$ is the complexified dual space of $\a$, then the
set of all joint eigendistributions of $\mathbb D(\Xi)$ is parametrized by
$\mathfrak a^*_c\times \mathcal D'(K/M)$.  More precisely, if we fix
$\lambda\in\mathfrak a^*_c$, then the joint eigenspace $\mathcal
D'_\lambda(\Xi)=\{\Psi\in\mathcal
D'(\Xi)\,|\,D_p\,\Psi=p(i\lambda-\rho)\,\Psi\;\text{for all}\; p\in
S(\mathfrak a) \}$ consists precisely of those distributions in $\Xi$ of the form
\begin{equation}\label{E:eigendist01}
  \Psi(\varphi)=\int_{K/M}^{}\int_{A}^{}\varphi(kM,a)\,e^{(i\lambda+\rho)(\log a)}\,da\,dS(kM)\qquad
(\varphi\in\mathcal D(\Xi))
\end{equation}
 for some $S\in \mathcal D'(K/M)$.

A {\it conical distribution} in $\Xi$ is an $MN$-invariant joint
eigendistributon of $\mathbb D(\Xi)$.  If $\lambda$ is regular and
simple, it turns out that the vector space of conical distributions is
$w$-dimensional, and an explicit basis $\{\Psi_{\lambda,s}\}$ can be
found in \cite{Duality}, each of which is supported in a Bruhat orbit
in $\Xi$.  For exceptional $\lambda$, the problem of
classification of the conical distributions turns out to be much more
difficult, although for rank one it is completely solved (\cite{Duality}, \cite{Hu}).

In this paper we consider the analogue of conical distributions on the
space of flat horocycles: these would be the translates, under the Cartan motion
group, of the tangent space at the origin $o$ in $X$
to the identity horocycle $\xi_0=N\cdot o$.

Explicitly, let us consider the Cartan motion group $G_0=K\ltimes \fk p$.
$G_0$ has group law $(k,X)\,(k',X')=(kk',X+k\cdot X')$, for $k,\,k'\in K$ and $X,\,X'\in\mathfrak p$, where we have
put $k\cdot X'=Ad\;k(X')$.  The mapping 
\begin{equation}\label{E:g-ident}
  (T,X)\mapsto T+X
\end{equation}
identifies
the Lie algebra $\mathfrak g_0$ of $G_0$ with $\mathfrak g$ as vector
spaces.  Under this identification, the adjoint representation $\text{Ad}_0$ of $G_0$ on $\mathfrak g_0$ is
given by
\begin{equation}\label{E:ad0}
  \text{Ad}_0\,(k,X)\,(T'+X')=\text{Ad}\,k\,(T')+k\cdot
X'-[\text{Ad}\,k\,(T'),X]
\end{equation}
and the Lie bracket $[\phantom{X},\phantom{Y}]_0$ on $\mathfrak g_0$ is given by
\begin{equation}\label{E:g0-bracket}
  [T+X,T'+X']_0=[T,T']+[T,X']-[T',X]
\end{equation}
with $T,\,T'\in\mathfrak k,\,
X,\,X'\in\mathfrak p$, where the Lie brackets on the right are taken
in $\mathfrak g$.  In effect, the Lie bracket on $\mathfrak g_0$ is the same as that on $\mathfrak g$, except that the subspace $\mathfrak p$ has been made abelian.

 Now $G_0$  acts transitively on $\fk
p$ by $(k,X)\cdot Y=X+k\cdot Y$, with $k\in K$ and $X,\,Y\in\p$.  Let
 $\fk q$ be the orthogonal complement of $\fk a$ in $\fk p$.  If we identify $\p$ with the tangent space $T_o X$, we have
$\mathfrak q= T_o(N\cdot o)$. A
{\it flat horocycle} is a translate of $\fk q$ by an element of $G_0$.
Let $\Xi_0$ be the set of all flat horocycles.  Then of course $\Xi_0$ is a
homogenous space of $G_0$, and its isotropy subgroup at $\fk q$ given by
$H=M'\ltimes \fk q$ (\cite{GASS}).

The flat horocycle Radon transform of  $f\in C_c(\mathfrak p)$ is
 the function on $\Xi_0$ defined by
$$
Rf(\xi)=\int_{\xi}^{}f(X)\,dm(X)\qquad (\xi\in \Xi_0)
$$
where $dm(X)$ is the Euclidean measure on $\xi$.  One may view this as
the flat analogue of the horocycle Radon transform on a noncompact
symmetric space $X$.  Properties of this
transform, such as an inversion formula, and range and support
theorems,  have been studied in several papers (\cite{Duality3},
\cite{GASS}, \cite{Orloff1}, \cite{Orloff2}).  

For each $s\in W$, we choose a representative $m_s\in M'$.
Then the map $\pi:K/M\times \fk a\to \Xi_0$ given by
$\pi(kM,H)=k\cdot(H+\fk q)$ is  $w$ to one, with
$\pi(kM,H)=\pi(km_s^{-1}M,sH)$. We can thus identify $\Xi_0$
with the associated bundle $K/M\times_W \fk a$ over $K/M'$, where $K/M$ can be
viewed as a principal bundle over $K/M'$ with discrete structure group
$W=M'/M$.  For convenience, we put $[kM,H]=\pi(kM,H)$.  It will be
clear from the context that this will not be confused with the Lie
bracket.

Using the above notation, the action of $G_0$ on $\Xi_0$
is given by
\begin{align}
  (k,X)\cdot [k_0M,H_0]&=X+k\cdot\left(k_0\cdot(H_0+\fk q)\right)\nonumber\\
&=kk_0\cdot\left( H_0+((kk_0)^{-1}\cdot X)_\fk a+\fk
  q\right)\nonumber\\
&=[kk_0M,H_0+((kk_0)^{-1}\cdot X)_\fk a]\label{E:gpaction1}
\end{align}
Here $X_\fk a$ is the orthogonal projection (under the Killing form)
of $X\in \fk p$ onto $\fk a$.

It will be convenient to note that $G_0$ also acts transitively on the
product manifold $\widetilde\Xi_0=K/M\times\fk a$ via
\begin{equation}\label{E:g0-action1}
   (k,X)\cdot
(k_0M,H_0)=(kk_0M,H_0+((kk_0)^{-1}\cdot X)_\fk a).
\end{equation}
  (That this is a
group action is straighforward to verify.)  We can think of
$\widetilde \Xi_0$ as the space of oriented flat horocycles in
$\mathfrak p$. 
 The isotropy subgroup of
$G_0$ at the origin
$\widetilde\xi_0=(eM,0)\in\widetilde\Xi_0$ is $M\ltimes\mathfrak
q$. From \eqref{E:gpaction1} and \eqref{E:g0-action1}, it is immediate that the
projection map $\pi:\widetilde\Xi_0\to\Xi_0$ commutes with the
action of $G_0$.  It will frequently be useful to do  calculations on
$\Xi_0$  by  lifting them up to $\widetilde \Xi_0$.  All groups
being unimodular, there are unique (up to constant multiple) 
$G_0$-invariant measures on $\Xi_0$ and on $\widetilde\Xi_0$, which we
can take in both cases to be $dk_M\,dH$.

\section{Invariant Differential Operators on $\Xi_0$ and $\widetilde\Xi_0$}
In this section we determine the algebras $\mathbb D(\Xi_0)$ and
$\mathbb D(\widetilde \Xi_0)$ of  $G_0$-invariant differential operators on $\Xi_0$ and
$\widetilde \Xi_0$, respectively.

All algebras here are over $\mathbb C$.
Let $I(\mathfrak p)$ and $I(\mathfrak a)$ be the subalgebras of $\text{Ad}\,K$-invariant
elements of $S(\p)$ and of $W$-invariant elements of $S(\mathfrak a)$,
respectively.
It is clear that the algebra $\mathbb D(\fk p)$ of $G_0$-invariant differential
operators on $\fk p$ is  $I(\fk p)$.
It is also a well-known fact that the restriction mapping $p\mapsto \overline p=p|_\mathfrak a$ is
an isomorphism of $I(\mathfrak p)$ onto $I(\mathfrak a)$.

Now let $P\in\mathcal S(\fk a)$.  Then from \eqref{E:g0-action1} the
differential operator $D_P$ on $\widetilde \Xi_0$ given by
\begin{equation}\label{E:constcoeff1}
  D_P\,\Phi(kM,H)=\partial(P)_H\,\Phi(kM,H)\qquad(\Phi\in\mathcal
E(\widetilde\Xi_0))
\end{equation}
is easily seen to belong to $\mathbb D(\widetilde\Xi_0)$.  If $P\in I(\fk a)$, we abuse
notation and also use $D_P$ to denote the (well-defined) differential operator on $\Xi_0$
given by
\begin{equation}\label{E:constcoeff2}
  D_P\,\varphi[kM,H]=\partial(P)_H\,\varphi[kM,H]\qquad(\varphi\in\mathcal E(\Xi_0))
\end{equation}
Then it follows from \eqref{E:gpaction1} that $D_P\in\mathbb D(\Xi_0)$.
For $\varphi\in\mathcal E(\Xi_0)$, put $\widetilde\varphi=\varphi\circ\pi$.
Then clearly
\begin{equation}\label{E:lift1}
  (D_P\varphi)^{\widetilde{}}=D_P\widetilde\varphi
\end{equation}

For $P\in\mathcal S(\fk a)$, we let $P^*$ be its formal adjoint in
$\fk a$.  Then the adjoint of the differential operator $D_P$ on
$\widetilde\Xi_0$ (with respect to the $G_0$-invariant measure $dk_M\,dH$) is $D_{P^*}$.  The same holds for the operator $D_P$ on
$\Xi_0$ if $P\in I(\fk a)$.

\begin{theorem}\label{T:invdiff1}
\begin{enumerate}
\item The map $P\mapsto D_P$ is an isomorphism of $\mathcal S(\fk a)$
  onto $\mathbb D(\widetilde\Xi_0)$.
\item The map $P\mapsto D_P$ is an isomorphism of $I(\fk a)$
  onto $\mathbb D(\Xi_0)$.
\end{enumerate}
\end{theorem}

This theorem is the flat analogue of  Theorem 2.2, Chapter II in
\cite{GASS}, which characterizes the algebra $\mathbb D(\Xi)$ of left
$G$-invariant differential operators on the horocycle space $\Xi$.
Our proof below is an adaptation of the proof of that theorem.

Let $H_0=M\ltimes\mathfrak q$, so that $\widetilde \Xi_0=G_0/H_0$.  We let 
$\mathfrak m$ denote the Lie algebra of $M$, and let
$\mathfrak l$ denote the orthogonal complement of $\mathfrak m$ in
$\mathfrak k$ with respect to $-B$.  Then we have the orthogonal decomposition
\begin{equation}\label{E:orthdecomp}
  \mathfrak g_0=\mathfrak g=(\mathfrak
m\oplus\mathfrak q)\oplus\mathfrak l\oplus\mathfrak a.
\end{equation}
Let $p:
g_0\mapsto g_0H_0$ be the coset map from $G_0$ onto $\widetilde
\Xi_0$, and let $\tau(g):g_0H_0\mapsto gg_0H_0$ be left translation by
$g\in G_0$ on $\widetilde\Xi_0$. Then $\tau(k,X)$ is given by
\eqref{E:g0-action1} and from that we have $p(k,X)=(kM,(k^{-1}\cdot
X)_\mathfrak a)$.  Now if $e_0=(e,0)$ is the identity
element of $G_0$, then \eqref{E:orthdecomp} shows that $dp_{e_0}$ is a linear bijection of $\mathfrak
l\oplus\mathfrak a$ onto the tangent space
$T_{\widetilde\xi_0}\widetilde\Xi_0$.  Let $\sigma$ be the orthogonal
projection of $\mathfrak g$ onto $\mathfrak l\oplus\mathfrak a$
according to the decomposition \eqref{E:orthdecomp}.
It is straightforward to show that
\begin{equation}\label{E:ad-commute}
  dp_{e_0}\circ\sigma\circ\text{Ad}_0\,(h)=d\tau(h)\circ dp_{e_0}\circ \sigma
\end{equation}
for all $h\in H_0$.  Thus the restriction of $dp_{e_0}$
to $\mathfrak l\oplus\mathfrak a$ intertwines the representations
$\sigma\circ\text{Ad}_0\,(h)$ and $d\tau(h)$ of $H_0$ on $\mathfrak
l\oplus\mathfrak a$ and on $T_{\widetilde\xi_0}\widetilde\Xi_0$, respectively.

While the pair $(G_0,H_0)$ is not reductive, it is nonetheless
possible to determine $\mathbb D(\widetilde\Xi_0)$ from the elements
of the (complexified) symmetric algebra $S(\mathfrak l\oplus\mathfrak
a)$ which are invariant under $\sigma\circ\text{Ad}\,(H_0)$.

\begin{lemma}\label{T:invariantpoly1}
$S(\mathfrak a)$ is precisely the algebra of elements in the symmetric algebra $S(\mathfrak
l\oplus\mathfrak a)$ which are invariant under
$\sigma\circ\text{Ad}_0\,(H_0)$.
\end{lemma}
\begin{proof}
Let $(m,X)\in H_0$. Then according to \eqref{E:ad0},  we have $\text{Ad}_0\,(m,X)\,(H)=H$ for any $H\in\mathfrak
a$.  This shows that $\mathfrak a$, and hence $S(\mathfrak a)$, is invariant
under $\text{Ad}_0\,(H_0)$ and thus also under $\sigma\circ\text{Ad}_0\,(H_0)$.

For the converse, let $\text{ad}_0$ denote the adjoint representation
on the Lie algebra $\mathfrak g_0$.  Then $\sigma\circ\text{ad}_0$ is
the representation of the Lie subalgebra $\mathfrak m\oplus\mathfrak q$ (of $\mathfrak g_0$) on $\mathfrak
l\oplus\mathfrak a$ corresponding to the representation
$\sigma\circ\text{Ad}_0$ of $H_0$ on the same space.  For convenience, for each $T+X\in\mathfrak
m\oplus\mathfrak q$, we let $d(T+X)$ denote the restriction of
$\sigma\circ\text{ad}_0(T+X)$ to $\mathfrak l\oplus\mathfrak a$. We
then extend $d(T+X)$ to a derivation of the symmetric algebra
$S(\mathfrak l\oplus\mathfrak a)$.

  We will
prove that if $Q\in S(\mathfrak l\oplus\mathfrak a)$ such that
\begin{equation}\label{E:annih}
  d(Y)\,Q=0\qquad\text{for all}\;Y\in\mathfrak q
\end{equation}
then $Q\in S(\mathfrak a)$.  This will then imply that the elements of
$S(\mathfrak l\oplus\mathfrak a)$ invariant under
$\sigma\circ\text{Ad}_0\,(\mathfrak q)$ belong to $S(\mathfrak a)$, which will
prove the lemma.

For each $\alpha\in\Sigma^+$, let
$X_1^\alpha,\cdots,X_{m_\alpha}^\alpha$ be an orthonormal basis of the
  restricted root space $\mathfrak g_\alpha$ with respect to the inner
  product $B_\theta(X,Y)=-B(X,\theta(Y))$ on $\mathfrak g$.  Then the
  vectors $E_i^\alpha=X_i^\alpha+\theta(X_i^\alpha)$ form an
  orthogonal basis (with respect to $-B$) of the
  subspace
$$
\mathfrak l_\alpha=\{T\in\mathfrak
k\,|\,\text{ad}\,(H)^2\,T=\alpha(H)^2\,T\;\text{for
  all}\;H\in\mathfrak a\}.
$$
of $\mathfrak k$. 
Likewise, the vectors $Y_i^\alpha=X_i^\alpha-\theta(X_i^\alpha)$ form
an orthogonal basis (with respect to $B$) of
$$
\mathfrak q_\alpha=\{X\in\mathfrak
p\,|\,\text{ad}\,(H)^2\,X=\alpha(H)^2\,X\;\text{for
  all}\;H\in\mathfrak a\}.
$$
Finally, we have $\mathfrak l=\oplus_{\alpha\in\Sigma^+}\mathfrak
l_\alpha$ and $\mathfrak q=\oplus_{\alpha\in\Sigma^+}\mathfrak
q_\alpha$.

If $\alpha\neq \beta$, it is easy to check that
$[Y_i^\alpha,E^\beta_j]_0=[Y_i^\alpha,E^\beta_j] \in\mathfrak q$ and therefore
$$
d(Y^\alpha_i)(E^\beta_j)=0\qquad (1\leq i\leq m_\alpha,\;1\leq j\leq m_\beta)
$$
On the other hand, for $1\leq i,j\leq m_\alpha$, 
\begin{align*}
  [Y^\alpha_i,E^\alpha_j]_0&=[Y^\alpha_i,E^\alpha_j]\\
&=([X^\alpha_i,X^\alpha_j]-\theta [X^\alpha_i,X^\alpha_j])+
([X^\alpha_i,\theta(X^\alpha_j)]-\theta [X^\alpha_i,\theta(X^\alpha_j)])
\end{align*}
The first quantity on the right above belongs to $\mathfrak q$.  If
$i\neq j$, then $[X^\alpha_i,\theta(X^\alpha_j)]\in\mathfrak m$, so
the second expression on the right above vanishes.  If $i=j$, then the
second quantity on the right equals $2A_\alpha$, where $A_\alpha$ is
the vector in $\mathfrak a$ such that $B(A_\alpha,H)=\alpha(H)$ for
all $H\in\mathfrak a$.

We conclude that
\begin{equation}\label{E:d-action}
  d(Y^\alpha_i)(E^\beta_j)=\begin{cases}
2A_\alpha&\qquad\text{if}\;\alpha=\beta\;\text{and}\;i=j\\
0&\qquad\text{otherwise}
\end{cases}
\end{equation}

Suppose now that $Q\in S(\mathfrak l\oplus\mathfrak a)$ such that
$d(Y)\,Q=0$ for all $X\in\mathfrak q$.  Fix any basis $H_1,\ldots,H_l$
of $\mathfrak a$.  Then $Q$ can be written uniquely as a polynomial in
the $E^\beta_j$ with coefficients in $S(\mathfrak a)$:
\begin{equation}\label{E:q-def1}
  Q=\sum_N \bigl(P_N(H_1,\ldots, H_l)\,\prod_{\beta\in\Sigma^+}
((E^\beta_1)^{n(\beta,1)}\cdots (E^\beta_{m_\beta})^{n(\beta,m_\beta)})\bigr)
\end{equation}
where the sum ranges over multiindices $N=(n(\beta,j))\;(1\leq j\leq
m_\beta,\,\beta\in\Sigma^+)$.  For convenience, let us put $E(\beta)^{N(\beta)}=(E^\beta_1)^{n(\beta,1)}\cdots
(E^\beta_{m_\beta})^{n(\beta,m_\beta)}$ and
$P_N=P_N(H_1,\ldots,H_l)$.

Since $d(Y^\alpha_i)\,H=0$ for all $H\in\mathfrak a$, \eqref{E:d-action} implies that
\begin{multline}\label{E:d-action1}
  d(Y^\alpha_i)\,Q=\\2A_\alpha\,\sum_{N\neq 0} n(\alpha,i)\,
 P_N\,\bigl(\prod_{\beta\neq \alpha}
E(\beta)^{N(\beta)}\,((E^\alpha_1)^{n(\alpha,1)}\cdots (E^\alpha_i)^{n(\alpha,i)-1}\cdots
(E^\alpha_{m_\alpha})^{n(\alpha,m_\alpha)})\bigr)
\end{multline}
Since the right hand side equals $0$, the coefficient
of $A_\alpha$ above must equal $0$.  This coefficient is therefore 
an empty sum.  Since $Y^\alpha_i$ is arbitrary, we conclude that there is only one summand in
\eqref{E:q-def1}, the one corresponding to $N=0$.   This shows that $Q\in S(\mathfrak a)$.
\end{proof}

Let us recall that, by definition, $H=M'\ltimes \mathfrak q$.
  
\begin{corollary}\label{T:invariantpoly2} The algebra of elements of
  $S(\mathfrak l\oplus\mathfrak a)$ invariant under
  $\sigma\circ\text{Ad}_0 (H)$ is $I(\mathfrak a)$.
\end{corollary}
\begin{proof}
This is clear from Lemma \ref{T:invariantpoly1} and \eqref{E:ad0}.
\end{proof}

The rest of the proof of the first assertion in Theorem \ref{T:invdiff1} proceeds exactly as
in Helgason's book (\cite{GASS}, Theorem 2.2, Chapter II).  For completeness,
we include it.  It will be convenient here to denote the basis $\{E^\alpha_i\}$ of $\mathfrak l$
by $H_{l+1},\ldots,H_{l+r}$.  Then for some $\delta>0$, the inverse of the map
$$
(t_1,\ldots,t_{l+r})\mapsto \exp(t_1H_1+\cdots+t_{l+r}H_{l+r})\,H_0\qquad(\sum t_i^2<\delta^2)
$$
is a chart on a neighborhood of the identity coset
$eH_0=\widetilde\xi_0$ in $\widetilde\Xi_0$.  Suppose that
$D\in\mathbb D(\widetilde\Xi_0)$.  Then there is a unique polynomial $P$ in
$l+r$ variables such that
\begin{equation}\label{E:local1}
  D\varphi(\widetilde\xi_0)=\left.P\left(\frac{\partial}{\partial
    t_1},\ldots,\frac{\partial}{\partial t_{l+r}}\right)\,\varphi(\exp (\sum t_iH_i)\cdot\widetilde\xi_0)\right|_{(t)=(0)}
\end{equation}
for all $\varphi\in\mathcal E(\widetilde\Xi_0)$.  Now for each $h\in
H_0$, there is a diffeomeorphism $(t_1,\ldots,t_{l+r})\mapsto
(s_1,\ldots,s_{l+r})$ on neighborhoods of $0\in\mathbb R^{l+r}$ such
that
$$
\tau(h)\,\exp\left(\sum t_iH_i\right)\,H_0=\exp \left(\sum s_jH_j\right)\,H_0.
$$
For convenience, let us put $\varphi(\exp(\sum
t_iH_i)\,H_0)=\varphi(t_1,\ldots,t_{l+r})$.  Since
$D(\varphi)(\widetilde\xi_0)=D(\varphi^{\tau(h)})(\widetilde\xi_0)$,
we have
\begin{equation}\label{E:p*}
  \left.P\left(\frac{\partial}{\partial
    t_1},\ldots,\frac{\partial}{\partial
    t_{l+r}}\right)\,\left(\varphi(t_1,\ldots,t_{l+r})-\varphi(s_1,\ldots,s_{l+r})\right)\right|_{(t)=(0)}
\end{equation}
Let $P^*$ denote the sum of the highest order terms in $P$, and write 
$$
P^*=\sum_{|J|=N} a_J 
\left(\frac{\partial}{\partial
  t_1}\right)^{j_1}\circ\cdots\circ\left(\frac{\partial}{\partial t_{l+r}}\right)^{j_{l+r}}.
$$
If we fix a multiindex $J$ of order $N$ and let
$\varphi(t_1,\ldots,t_{l+r})=t^J=t_1^{j_1}\ldots t_{l+r}^{j_{l+r}}$ near the origin,
then \eqref{E:p*} shows that
\begin{equation}\label{E:p*2}
  a_J=\sum_{|I|=N} R_{JI}\,a_I
\end{equation}
where $(R_{JI})$ is the matrix of the linear operator on the vector
space of homogeneous degree $N$ polynomial functions on $\mathbb
R^{l+r}$ extending the operator on $\mathbb R^{l+r}$ whose matrix is
the Jacobian matrix $(\partial s_j/\partial t_i)$ at $(t)=(0)$.  But
this Jacobian matrix is also the matrix of $\sigma\circ
\text{Ad}_0\,(h)$ with respect to the basis $\{H_i\}$ of $\mathfrak
l\oplus\mathfrak a$.  Equation \eqref{E:p*2} then shows that $\sum_{|J|=N}
a_J H^J$ is invariant under $\sigma\circ
\text{Ad}_0\,(h)$.  Thus by Lemma \ref{T:invariantpoly1}, we conclude that
$P^*=P^*(\partial/\partial t_1,\ldots,\partial/\partial t_l)$.  Since
$D$ is $G_0$-invariant, we see that
$$
D\varphi(g_0\cdot\widetilde\xi_0)=P^*\left(\frac{\partial}{\partial
    t_1},\ldots,\frac{\partial}{\partial
    t_l}\right)\,\varphi(g_0\,\exp(\sum_{i=1}^l
t_iH_i)\cdot\widetilde\xi_0)\biggl|_{(0)}+\;
\text{lower order terms}
$$
so that  $D-D_{P^*}$ is an element of $\mathbb D(\widetilde\Xi_0)$
whose order is less than the order of $D$.  A
simple induction on the order then completes the proof of the first
assertion of Theorem \ref{T:invdiff1}.

For the second assertion, suppose that $D\in\mathbb D(\Xi_0)$.  Then
there exists a polynomial $P$ such that \eqref{E:local1} holds for all functions $\varphi\in\mathcal E(\Xi_0)$, with
$\xi_0$ replacing $\widetilde\xi_0$.  With this substitution, the rest of the proof above
carries over, with $h\in H_0$ replaced by $h\in H=M'\ltimes\mathfrak
q$, and with $P^*(H_1,\ldots,H_l)$ $M'$-invariant by Corollary
\ref{T:invariantpoly2}.

\section{The Space of Joint Eigendistributions}
Suppose that $\Psi\in\mathcal D'(\Xi_0)$ is an eigendistribution of
$\mathbb D(\Xi_0)$.  Then according to Lemma 3.11, Chapter III of
\cite{GGA}, there exists $\lambda\in\fk a_c^*$ (unique up to $W$-orbit) such
that
\begin{equation}\label{E:eigendist1}
  D_P\Psi=P(i\lambda)\,\Psi
\end{equation}
for all $P\in I(\fk a)$.  We let $\mathcal D'_\lambda(\Xi_0)$ denote
the vector space consisting of all $\Psi\in\mathcal D'(\Xi_0)$
satisfying \eqref{E:eigendist1}.

Any eigendistribution $\Psi\in\mathcal D'(\widetilde\Xi_0)$ of
$\mathbb D(\widetilde\Xi_0)$ likewise corresponds to a unique $\lambda\in\fk
a_c^*$ satisfying \eqref{E:eigendist1} for all $P\in\mathcal S(\fk
a)$. For such $\lambda$, we denote the vector space of all such
distributions by $\mathcal D'_\lambda(\widetilde\Xi_0)$.

The following can be proved in a manner analogous to the proof of
Proposition 4.4, Chapter II in \cite{GASS}.

\begin{proposition}\label{T:eigendist2}
Let $\Psi\in\mathcal D'_\lambda(\widetilde\Xi_0)$.  Then there is a unique
$S\in\mathcal D'(K/M)$ such that
\begin{equation}\label{E:eigendist3}
  \Psi(\varphi)=\int_{K/M}^{}\int_{\fk a}^{}\varphi(kM,H)\,e^{i\lambda(H)}\,dH\,dS(kM).
\end{equation}
Conversely, if $S\in\mathcal D'(K/M)$, then the distribution $\Psi$ on
$\widetilde\Xi_0$ defined above belongs to $\mathcal
D'_\lambda(\widetilde\Xi_0)$. 
\end{proposition}

If $F\in\mathcal E(\widetilde\Xi_0)$, we define $F_\pi\in\mathcal
E(\Xi_0)$ by
$$
F_\pi[kM,H]=\frac{1}{w}\sum_{s\in W} F(km_s^{-1}M, m_s\cdot H)
$$
Then the pullback $\widetilde\Phi$ of a distribution $\Phi\in\mathcal
D'(\Xi_0)$ is defined by
\begin{equation}\label{E:pullback}
  \widetilde\Phi(F)=\Phi(F_\pi)\qquad\qquad(F\in\mathcal D(\widetilde\Xi_0))
\end{equation}
Note that
$$
\widetilde\Phi(\widetilde\varphi)=\Phi(\varphi)
$$
for all $\Phi\in\mathcal D'(\Xi_0),\;\varphi\in\mathcal D(\Xi_0)$.
Let $P\in I(\fk a)$ and  $\Phi\in\mathcal D'(\Xi_0)$.  Then 
it is easy to see from \eqref{E:constcoeff1} and \eqref{E:constcoeff2}
and the fact that $D_P(F_\pi)=(D_P F)_\pi$,
that, in analogy with \eqref{E:lift1}, we have
\begin{equation}\label{E:lift2}
  \left(D_P\Phi\right)^{\widetilde{}}=D_P\widetilde\Phi.
\end{equation}

Since $\mathbb D(\Xi_0)$ is smaller than $\mathbb
D(\widetilde\Xi_0)$, it is not true that $\widetilde\Phi$ belongs to
$\mathcal D'_\lambda(\widetilde\Xi_0)$ whenever $\Phi\in\mathcal
D'_\lambda(\Xi_0)$.  (It is easy to construct smooth counterexamples.)
Nonetheless, we can obtain a result similar to Proposition
\ref{T:eigendist2} above, as follows.

Suppose that $\Phi\in\mathcal D'_\lambda(\Xi_0)$.  Then from \eqref{E:lift2}
$\widetilde\Phi$ satisfies
\begin{equation}\label{E:eigenfcn1}
  D_P(\widetilde\Phi)=P(i\lambda)\,\widetilde\Phi\qquad\qquad(P\in I(\fk a))
\end{equation}

Now for functions $\alpha\in\mathcal D(\fk a)$ and
$\beta\in\mathcal E(K/M)$, let $\beta\otimes\alpha$ be the function
$\beta(kM)\,\alpha(H)$ on $\widetilde\Xi_0=K/M\times\fk a$.  The
linear span of such functions is dense in $\mathcal
D(\widetilde\Xi_0)$.

If we fix $\beta\in\mathcal E(K/M)$, the map 
\begin{equation}\label{E:restr1}
  T_\beta:\alpha\in\mathcal D(\fk a)\to \widetilde\Phi(\beta\otimes\alpha)
\end{equation}
is a distribution in $\fk a$; in fact, we see from \eqref{E:eigenfcn1}
that $T_\beta$ is an eigendistribution
of the algebra $I(\fk a)$.  Since this algebra contains elliptic
elements, it follows that $T_\beta$ is in fact a smooth eigenfunction of
$I(\fk a)$, with
\begin{equation}\label{E:restr2}
  (\partial(P)\,T_\beta)(H)=P(i\lambda)\,T_\beta (H)
\end{equation}
for all $P\in I(\fk a)$.  The space of such eigenfunctions is
described in \cite{GGA}, Chapter III, Theorem 3.13.  Let $W_{\lambda}$
denote the subgroup of $W$ consisting of those elements fixing
$\lambda$, let $I_{\lambda}(\fk a)$ be the subalgebra of
$W_{\lambda}$-invariant elements of $\mathcal S(\fk a)$, and let
$H_{\lambda}$ be the vector space of $W_{\lambda}$-harmonic polynomial functions
on $\fk a$.

Then for each element $s\lambda$ in the orbit $W\cdot \lambda$, there exists
a unique polynomial $P_{s\lambda}(\beta)(H)$ in $H_{s\lambda}$, with
coefficients depending on $\beta$, such that
\begin{equation}\label{E:restr3}
  T_\beta(H)=\sum_{s\lambda\in W\cdot \lambda} P_{s\lambda}(\beta)(H)\,e^{is\lambda(H)}
\end{equation}
for all $H\in\fk a$.  When $\lambda$ is regular, the
$P_{s\lambda}(\beta)$ are just constants (depending, of course, on $\beta$).

For fixed $H\in\fk a$, the map $\beta\in\mathcal E(K/M)\to
P_{s\lambda}(\beta)(H)$ is continuous, and from this it is not hard to see that the coefficients of the polynomials
$P_{s\lambda}(\beta)(H)$ are distributions on $K/M$.  
More precisely,
for each $s\lambda$, fix a basis $P_{s\lambda,j}(H)$ ($1\leq j\leq
r=|W_{\lambda}|$) of $H_{s\lambda}$.  Then 
\begin{equation}\label{E:restr4}
  P_{s\lambda}(\beta)(H)=\sum_{j=1}^{r} S_{s\lambda,j}(\beta)\,P_{s\lambda,j}(H)
\end{equation}
Each coefficient $S_{s\lambda,j}$ is a distribution on $K/M$ uniquely
determined, of course, by the choice of the basis $\{P_{s\lambda,j}\}$.  Hence,
by \eqref{E:restr1} \eqref{E:restr3}, and \eqref{E:restr4}, we see that
\begin{equation}\label{E:phi1}
  \widetilde\Phi(F)=\sum_{s\lambda\in
    W\cdot\lambda}\,\sum_{j=1}^{r}
\int_{K/M}^{}\int_{\fk a}^{}P_{s\lambda,j}(H)\,F(kM,H)\,e^{is\lambda(H)}\,dH\,dS_{s\lambda,j}(kM)
\end{equation}
for all $F\in\mathcal D(\widetilde\Xi_0)$ of the form
$\beta\otimes\alpha$.  Since the $\beta\otimes\alpha$ span a dense
subspace of $\mathcal D(\widetilde\Xi_0)$, formula \eqref{E:phi1} holds for
all $F\in\mathcal D(\widetilde\Xi_0)$.

When $\lambda$ is regular, each $H_{s\lambda}=\mathbb C$ (so
we can take $1$ as its basis), and the formula above reduces to
\begin{equation}\label{E:phi2}
  \widetilde\Phi(F)=\sum_{s\in W}\int_{K/M}^{}\int_{\fk a}^{} F(kM,H)\,e^{is\lambda(H)}\,dH\,dS_{s\lambda}(kM)
\end{equation}
for all $F\in\mathcal D(\widetilde\Xi_0)$.

We now proceed to obtain a more explicit characterization of the eigendistribution $\Phi\in\mathcal
D'_\lambda(\Xi_0)$.  For this, we note that expression
\eqref{E:restr4} shows that $P_{s\lambda}$ can be considered as an
element of $\mathcal D'(K/M)\otimes H_{s\lambda}$, with
$P_{s\lambda}=\sum_{j=1}^m S_{s\lambda,j}\otimes P_{s\lambda,j}$, so
that \eqref{E:phi1} becomes
\begin{equation}\label{E:phi3}
  \widetilde\Phi(F)=\sum_{s\lambda\in W\cdot \lambda} \int_{\fk
    a}^{}\int_{K/M}^{}F(kM,H)\,e^{is\lambda(H)}\,dP_{s\lambda}(kM)(H)\,dH
\end{equation}
We observe that by \eqref{E:restr4}, each $P_{s\lambda}$ is uniquely determined by $\Phi$.

Now the Weyl group $W$ acts (freely) on both $K/M$ and on
$\widetilde\Xi_0=K/M\times\fk a$ by $s\cdot kM=km_s^{-1}M$ and
$s\cdot(kM,H)=(km_s^{-1}M,sH)$.  Thus for each $t\in W$,
\begin{align}
  \widetilde\Phi(F)&=\widetilde\Phi^t(F)\nonumber\\
&=\sum_{s\lambda\in W\cdot \lambda} \int_{\fk
    a}^{}\int_{K/M}^{}F(km_t^{-1}M,t\cdot
  H)\,e^{is\lambda(H)}\,dP_{s\lambda}(kM)(H)\,dH\nonumber\\
&=\sum_{s\lambda\in W\cdot \lambda} \int_{\fk
    a}^{}\int_{K/M}^{}F(kM,t\cdot
  H)\,e^{is\lambda(H)}\,dP_{s\lambda}^t(kM)(H)\,dH\label{E:phi4}
\end{align}
where we have put $P_{s\lambda}^t=\sum_j S_{s\lambda,j}^t\otimes
P_{s\lambda,j}$.  Then right hand side of \eqref{E:phi4} then equals
\begin{equation}\label{E:phi5}
  \sum_{s\lambda\in W\cdot \lambda} \int_{\fk
    a}^{}\int_{K/M}^{}F(kM,H)\,e^{its\lambda(H)}\,t\cdot dP_{s\lambda}^t(kM)(H)\,dH
\end{equation}
where now $t\cdot P_{s\lambda}^t=\sum_j S_{s\lambda,j}^t\otimes
(t\cdot P_{s\lambda,j})$, an element of $\mathcal D'(K/M)\otimes
H_{ts\lambda}$.  By the uniqueness of the $P_{s\lambda}$, it follows that
$$
  P_{ts\lambda}=t\cdot P_{s\lambda}^t
$$
for all $s,t\in W$.  In particular,
$$
P_{s\lambda}=s\cdot P_\lambda^s \qquad\qquad (s\in W)
$$

Hence, for any $\varphi\in\mathcal D(\Xi_0)$, we have
\begin{align}
  \Phi(\varphi)&=\widetilde\Phi(\widetilde\varphi)\nonumber\\
&=\sum_{s\lambda\in W\cdot \lambda} \int_{\fk
    a}^{}\int_{K/M}^{}\widetilde\varphi(kM,H)\,e^{is\lambda(H)}\,s\cdot dP_{\lambda}^s(kM)(H)\,dH\nonumber\\
&=\sum_{s\lambda\in W\cdot \lambda} \int_{\fk
    a}^{}\int_{K/M}^{}\widetilde\varphi(km_s^{-1}M,s\cdot
  H)\,e^{i\lambda(H)}\,dP_{\lambda}(kM)(H)\,dH\nonumber\\
&=|W\cdot\lambda|\,\int_{\fk
    a}^{}\int_{K/M}^{}\widetilde\varphi(kM,
  H)\,e^{i\lambda(H)}\,dP_{\lambda}(kM)(H)\,dH\nonumber\\
&=|W\cdot\lambda|\,\int_{\fk
    a}^{}\int_{K/M}^{}\varphi[kM,
  H]\,e^{i\lambda(H)}\,dP_{\lambda}(kM)(H)\,dH\label{E:phi-value}
\end{align}

We are led to the following theorem.

\begin{theorem}\label{T:geneigendist}
Suppose that $\lambda\in\fk a^*_c$ and that $\Phi\in\mathcal
D'_\lambda(\Xi_0)$.  Then there exists a unique element
$Q_\lambda\in\mathcal D'(K/M)\otimes H_\lambda$ such that
\begin{equation}\label{E:phi6}
  \Phi(\varphi)=\int_{\fk
    a}^{}\int_{K/M}^{}\varphi[kM,
  H]\,e^{i\lambda(H)}\,dQ_{\lambda}(kM)(H)\,dH
\end{equation}
Conversely, given any element $Q_\lambda\in\mathcal D'(K/M)\otimes
H_\lambda$, the expression \eqref{E:phi6} defines a distribution $\Phi\in\mathcal D'_\lambda(\Xi_0)$.
\end{theorem}
{\it Remarks:}
\begin{enumerate}
\item Fix a basis $P_1,\ldots,P_r$ of $H_\lambda$.  (We may choose
  this basis to have real coefficients.)  If $\Phi\in\mathcal
  D'_\lambda(\Xi_0)$, the theorem above says that there exist unique distributions $T_j$ on
$K/M$ such that
\begin{equation}\label{E:eigendist16}
  \Phi(\varphi)=\sum_{j=1}^r \int_{K/M}^{}\int_{\fk
  a}^{} P_j(H)\,\varphi[kM,H]\,e^{i\lambda(H)}\,dH\,dT_j(kM)
\end{equation}
for all $\varphi\in\mathcal D(\Xi_0)$.
Conversely, for any distributions $T_j$ on $K/M$, the right hand side
of \eqref{E:eigendist16} defines a distribution $\Phi\in\mathcal D'_\lambda(\Xi_0)$.
\item Because of the ambiguity in the argument of $\varphi[kM,H]$, one
  should, strictly speaking, write equation \eqref{E:eigendist16} as
$$
\Phi(\varphi)=\sum_{j=1}^m \int_{K/M}^{}\int_{\fk
  a}^{} P_j(H)\,\widetilde\varphi(kM,H)\,e^{i\lambda(H)}\,dH\,dT_j(kM)
$$
\item Equation \eqref{E:eigendist16}  can also be written as
\begin{equation}\label{E:eigendist15}
\Phi(\varphi)=\sum_{j=1}^m  \int_{K/M}^{}\partial(P_j^*)\,\varphi^*[kM,\lambda]\,dT_j(kM),
\end{equation}
where $\varphi^*$ is the (well-defined) Fourier-Laplace transform on
$\Xi_0$:
$$
\varphi^*[kM,\lambda]=\int_{\mathfrak
  a}^{}\varphi[kM,H]\,e^{i\lambda(H)}\,dH\qquad
([kM,\lambda]\in K/M\times_W \mathfrak a^*_c)
$$
\end{enumerate}

\begin{proof}
Equation \eqref{E:phi6} follows from \eqref{E:phi-value} by putting
$Q_\lambda=|W\cdot\lambda|\,P_\lambda$.  The uniqueness of $Q_\lambda$
is a consequence of the uniqueness of the $P_{s\lambda}$, and in particular, of $P_\lambda$.

Conversely, suppose that $Q_\lambda\in\mathcal D'(K/M)\otimes
H_\lambda$.  If we fix a basis $P_1,\ldots, P_m$ of $H_\lambda$, we
can, as in Remark (1) above, write $Q_\lambda=\sum_j S_j\otimes
P_j$. The distribution $\Phi$ in \eqref{E:phi6} is then given by \eqref{E:eigendist16}, and 
thus we need to prove that the right hand side of \eqref{E:eigendist16}
defines a distribution $\Phi\in\mathcal D'_\lambda(\Xi_0)$.   Now the product
$P(H)\,e^{i\lambda(H)}$ belongs to the joint eigenspace $\mathcal
E_{i\lambda}(\fk a)=\{\alpha\in\mathcal E(\mathfrak
a)\,|\,\partial(P)\,\alpha=P(i\lambda)\,\alpha\;\text{for all}\;P\in I(\mathfrak a)\}$.
Hence for any $Q\in I(\fk a)$, we have
\begin{align*}
  (D_Q(\Phi))(\varphi)&=\sum_{j=1}^m \int_{K/M}^{}\int_{\fk
  a}^{}
\partial(Q^*)\varphi[kM,H]\,P_j(H)\,e^{i\lambda(H)}\,dH\,dT_j(kM)\\
&=Q(i\lambda)\,\sum_{j=1}^m \int_{K/M}^{}\int_{\fk
  a}^{} \varphi[kM,H]\,P_j(H)\,e^{i\lambda(H)}\,dH\,dT_j(kM)\\
&=Q(i\lambda)\,\Phi(\varphi),
\end{align*}
for all $\varphi\in\mathcal D(\Xi_0)$, proving the theorem.
\end{proof}

\begin{corollary}\label{T:l-regular}
Suppose that $\lambda$ is regular.  Then there is a linear bijection
from $\mathcal D'(K/M)$ onto $\mathcal D'_\lambda(\Xi_0)$ given by
\begin{align}
  T&\mapsto\Phi\nonumber\\
\Phi(\varphi)&=\int_{K/M}^{}\int_{\fk
  a}^{}\varphi[kM,H]\,e^{i\lambda(H)}\,dH\,dT(kM)\nonumber\\
&=\int_{K/M}^{}\varphi^*[kM,\lambda]\,dT(kM)\label{E:eigendist17}
\end{align}
\end{corollary}

\section{Conical Distributions}
By definition, a {\it conical distribution} on $\Xi_0$ is an
$H$-invariant eigendistribution of $\mathbb D(\Xi_0)$, where, as we
recall, $H$ is the isotropy subgroup of $G_0$ fixing $\fk q$:
$H=M'\ltimes \fk q$.

Suppose that $\Phi$ is a conical distribution on $\Xi_0$ belonging to
$\mathcal D'_\lambda(\Xi_0)$.  ($\lambda$ is of course determined up
to $W$-orbit.)  
First, for simplicity, let us assume that $\lambda$ is regular.  
Then we see that  $\Phi$ satisfies
\eqref{E:eigendist17}, for  unique $T\in\mathcal D'(K/M)$. 

In order to determine this distribution $T$ more explicitly, we first prove that the
collection of functions on $K/M$ given by
$\{\varphi^*[kM,\lambda]\,|\,\varphi\in\mathcal D(\Xi_0)\}$ equals
$\mathcal E(K/M)$.

For this, we first prove the following easy lemma.

\begin{lemma}\label{T:fcnalindep}
For $f\in\mathcal D(\fk a)$ and $\gamma\in\mathcal E(\fk a)$, put 
$(f,\varphi)=\int_{\fk a}^{}f(H)\,\varphi(H)\,dH$.  Suppose that 
$\gamma_1,\ldots,\gamma_m$ are linearly independent elements of
$\mathcal E(\fk a)$.  Then there exist functions $f_1,\ldots,f_m$ in
$\mathcal D(\fk a)$ such that the $m\times m$ matrix $((f_i,\gamma_j))$ is nonsingular.
\end{lemma}
\begin{proof}
This is an easy induction on $m$, the result being obviously true for
$m=1$.  Now suppose that the result is true for  $m-1$, and let
$\gamma_1,\ldots,\gamma_m$ be linearly independent elements of
$\mathcal E(\fk a)$. Now for any $f_1,\ldots,f_m\in\mathcal D(\fk a)$,
we expand the determinant of $((f_j,\gamma_k))$ by minors along the first
row.  The result is the expression
\begin{equation}\label{E:fcnprod}
  (f_1,A_1\gamma_1+\cdots+A_m\gamma_m)
\end{equation}
where $A_i$ is the $(1,i)$ minor of $((f_j,\gamma_k))$.  By the
induction hypothesis, we may choose $f_2,\ldots,f_m\in\mathcal D(\fk
a)$ such that $A_1\neq 0$.  Then, by linear independence, the second
argument of \eqref{E:fcnprod} is a nonzero smooth function on $\fk a$.
Hence, there exists $f_1\in\mathcal D(\fk a)$ such that
\eqref{E:fcnprod} is not equal to zero.
\end{proof}

This easily implies that there exist functions $f_j\in\mathcal D(\fk
a)$ such that $((f_i,\gamma_j))$ is any prescribed $m\times m$ matrix.
For every $h\in\mathcal D(\mathfrak a)$, let $h^*$ denote its
Fourier-Laplace transform
$$
h^*(\lambda)=\int_{\mathfrak a}^{}h(H)\,e^{i\lambda(H)}\,dH\qquad (\lambda\in\mathfrak a^*_c)
$$

\begin{lemma}\label{T:Ksurj}
Let $\lambda\in\fk a^*_c$ be regular.  Let $R$ be the linear map from
$\mathcal D(\Xi_0)$ to $\mathcal E(K/M)$ given by
$R\varphi(kM)=\varphi^*[kM,\lambda]$.  Then $R$ is onto.
\end{lemma}
\begin{proof}
The proof requires some care since $\Xi_0$ is not the product manifold
$K/M\times\fk a$ but a quotient of it.  We will first prove that for
any distinct elements $\lambda_1,\ldots,\lambda_m\in\fk a^*_c$, and
 any functions $\beta_1,\ldots,\beta_m\in\mathcal E(K/M)$, there
exists a function $F\in\mathcal D(\widetilde\Xi_0)$ such
that $F^*(kM,\lambda_j)=\beta_j(kM)$ for all $k\in K$ and all $j$.

Now the functions $e^{i\lambda_1},\ldots,e^{i\lambda_m}$ are
linearly independent elements of $\mathcal E(\fk a)$.  Hence by Lemma
\ref{T:fcnalindep} there exist
functions $h_1,\ldots,h_m\in\mathcal D(\fk a)$ such that the $m\times
m$ matrix $(h_i^*(\lambda_j))$ is nonsingular. Thus the system
$$
\left(\begin{array}{ccc}
  h_1^*(\lambda_1)&\cdots&h_m^*(\lambda_1)\\
&\ddots&\\
h_1^*(\lambda_m)&\cdots&h_m^*(\lambda_m)
\end{array}\right)\,\left(\begin{array}{c}
  F_1(kM)\\ \vdots\\ F_m(kM)
\end{array}\right)=\left(
\begin{array}{c}
  \beta_1(kM)\\ \vdots\\ \beta_m(kM)
\end{array}
\right)
$$
has smooth solutions $F_1(kM),\ldots,F_m(kM)$.  Putting
$F(kM,H)=\sum_j F_j(kM)\,h_j(H)$, we see that $F\in\mathcal
D(\widetilde\Xi_0)$ and $F^*(kM,\lambda_j)=\beta_j(kM)$.

Now fix $\beta\in\mathcal E(K/M)$.  We will prove that there exists a
$\varphi\in\mathcal D(\Xi_0)$ such that
$\varphi^*[kM,\lambda]=\beta(kM)$ for all $kM\in K/M$.

From the above we know that there exists a function
$F\in\mathcal D(\widetilde\Xi_0)$ such that
$F^*(kM,s\lambda)=\beta(km_sM)$ for all $kM\in K/M$ and all $s\in W$.
(Here $m_s\in M'$ is any coset representative of $s$.)  Put
$\varphi=F_\pi$, so that $\varphi\in\mathcal D(\Xi_0)$.
Then $\varphi^*[kM,\mu]=(1/w)\cdot\sum_{s\in W} F^*(km_s^{-1}M,s\mu)$ for all
$kM\in K/M$ and all $\mu\in\fk a^*_c$.  In particular, 
\begin{align*}
  \varphi^*[kM,\lambda]&=\frac{1}{w}\,\sum_{s\in W}
  F^*(km_s^{-1}M,s\lambda)\\
&=\beta(kM)
\end{align*}
for all $kM\in K/M$.
\end{proof}

Resuming our investigation of conical distributions, let us assume, 
as before, that $\Phi$ is conical distribution in $\mathcal
D'_\lambda(\Xi_0)$, where $\lambda$ is a fixed regular element in $\fk
a^*_c$.  Let $T$ be the unique element of $D'(K/M)$ given by \eqref{E:eigendist17}.

The $M'$ invariance of $\Phi$ implies that
\begin{equation}\label{E:m'-inv}
  \int_{K/M}^{}\varphi^*[m'kM,\lambda]\,dT(kM)=\int_{K/M}^{}\varphi^*[kM,\lambda]\,dT(kM)
\end{equation}
for all $m'\in M'$.  By Lemma \ref{T:Ksurj}, the functions
$\varphi^*[kM,\lambda]$ run through $\mathcal E(K/M)$ as $\varphi$
runs through $\mathcal D(\Xi_0)$.  Thus \eqref{E:m'-inv} shows that
$T$ is a left $M'$-invariant distribution on $K/M$.

The $\fk q$-invariance of $\Phi$ then shows that
\begin{align}
  \int_{K/M}^{}\varphi^*[kM,\lambda]\,dT(kM)&=
\int_{K/M}^{}\varphi^*[kM,\lambda]\,e^{-i\lambda((k^{-1}X)_\fk
  a)}\,dT(kM)\nonumber\\
&=\int_{K/M}^{}\varphi^*[kM,\lambda]\,e^{-iB(kA_\lambda,X)}\,dT(kM)\label{E:qinv1}
\end{align}
By Lemma \ref{T:Ksurj}, this implies that 
\begin{equation}\label{E:q-inv}
  T=e^{-iB(k\cdot A_\lambda,X)}\,T
\end{equation}
for all $k\in K$ and all $X\in\fk q$.

We will now prove that the property \eqref{E:qinv1} implies that $T$
has support in the discrete subset $M'/M$ of $K/M$.  For
this, consider any $k_0\in K\setminus M'$.  Since $\lambda$ is
regular, $k_0\cdot A_\lambda\notin\fk a^*_c$. It is easy to see that there exists
$X\in\fk q$ such that $B(k_0\cdot A_\lambda,X)\notin 2\pi \mathbb Z$.
(This is done by scaling $X$ if necessary.)  Fixing this $X$, there  exists a
neighborhood $U$ of $k_0M$ in $K/M$ such that $B(k\cdot
A_\lambda,X)\notin 2\pi \mathbb Z$ for all $kM\in U$. Hence the
function 
$kM\mapsto e^{iB(k\cdot A_\lambda,X)}-1$ is never $0$ on $U$, whereas by \eqref{E:q-inv}
 the distribution
$(e^{iB(k\cdot A_\lambda,X)}-1)\,T$ on $K/M$ vanishes.
This implies that $T=0$ on $U$.  Since $k_0M$ was chosen arbitrarily
in $K/M\setminus M'/M$, this proves that $T$ has support in the discrete set $M'/M$.

In particular, $T$ has the form 
\begin{equation}\label{E:discrete}
  T=\sum_{s\in W} D_s\,\delta_{m_sM}
\end{equation}
where $D_s$ is a linear differential operator on $K/M$.  We will now
prove that in fact
\begin{equation}\label{E:deltadist}
  T=c\,\sum_{s\in W} \delta_{m_sM}
\end{equation}
for some constant $c$.  For this, it suffices to prove that near the
identity coset $eM$ of $K/M$, $T$ is a multiple of the delta function
at $eM$.  That is to say, it suffices to prove that for all smooth
functions $\beta$ on $K/M$ supported on a small neighborhood of $eM$,
then $T(\beta)=c\,\beta(eM)$.  The $M'$-invariance of $T$ then proves
\eqref{E:deltadist}.

To this end, we introduce local coordinates on $K/M$ near $eM$.
Let $T_1,\ldots,T_{m_\alpha}$ be an orthonormal basis (with
respect to $-B$) of $\fk l_\alpha$.  (We could use
$T^\alpha_i=2^{-1/2}\,E^\alpha_i$ from the proof of Lemma
\ref{T:invariantpoly1}.)   The collection $\{T_j\}_{1\leq
  j\leq m_\alpha, \alpha\in \Sigma^+}$ is then an orthonormal basis of
$\fk l$.  We list these basis elements as $T_1,\ldots, T_r$ and assume
that $T_j$ belongs to the generalized eigenspace $\fk k_{\alpha_j}$.
Then the map
\begin{equation}\label{E:chart}
  \exp(t_1\,T_1+\cdots+t_r\,T_r)M\mapsto (t_1,\ldots,t_r)
\end{equation}
defines a chart on a neightborhood $U$ of $eM$ in $K/M$.  We assume that $U\cap M'/M=\{eM\}$.

For each $j$ let us put
$X_j=-i(B(\alpha_j,\lambda))^{-2}\,\text{ad}(A_\lambda)\,T_j$.  Since
$\lambda$ is regular, $X_j$ is well defined, and it is easy to see that
$X_1,\ldots,X_r$ is a basis of $\fk q^c$, orthogonal with respect to the Killing form.

Now suppose that $\beta$ is a smooth function on $K/M$ with support in
$U$.  Then by \eqref{E:discrete}, we have
\begin{equation}\label{E:dist1}
  T(\beta)=\sum_J c_J D^J\beta(0)
\end{equation}
where the sum runs through a finite collection of multiindices
$J=(j_1,\ldots,j_r)$, the $c_J$ are constants, 
and $D^J=\partial^{j_1+\cdots+j_r}/\partial t_1^{j_1}\cdots\partial
t_r^{j_r}$.  

In the sum \eqref{E:dist1}, we claim that $c_J=0$ when $|J|>0$.  Then
of course $T(\beta)=c_0\,\beta(0)$, and this will prove
\eqref{E:deltadist}.  To prove this, let us assume, to the contrary,
that $c_J\neq 0$ for some $J\neq 0$. Let $N=\max\{|J|\,|\,c_J\neq 0\}$.
Now by  \eqref{E:q-inv} we have
\begin{equation}\label{E:q-inv2}
  \sum_J c_J D^J\beta(0)=
\sum_J c_J D^J\left(e^{-iB(\exp(t_1T_1+\cdots+t_rT_r)\cdot A_\lambda,X)}\,\beta\right)(0)
\end{equation}
for all $X\in\fk q$ and all smooth functions $\beta$ supported in $U$.
In \eqref{E:q-inv2} choose a $\beta$ which is identically $1$ on a small
neighborhood of $0$.  Then the left hand side of \eqref{E:q-inv2} is
$c_0$.  On the other hand, if we write $X=z_1\,X_1+\cdots+z_r\,X_r$,
where $z_j\in\mathbb C$, then the right hand
side is 
$$
\sum_J c_J D^J\left(e^{-iB(\exp(t_1T_1+\cdots+t_rT_r)\cdot A_\lambda,X)}\right)(0),
$$
a polynomial of degree $N$ in $z_1,\ldots,z_r$.  Its homogeneous
component of degree $N$ equals
$$
\sum_{|J|=N} c_J\,z^J,
$$
where we have put $z^J=z_1^{j_1}\cdots z_r^{j_r}$ when
$J=(j_1,\ldots,j_r)$.  This yields a contradiction.  We obtain the
following result.

\begin{theorem}\label{E:regularconical}
Suppose that $\lambda\in \fk a^*_c$ is regular.  Then the space of
conical distributions in $\mathcal D'_\lambda(\Xi_0)$ is
one-dimensional, with basis given by $\Phi_\lambda$, where
\begin{equation}\label{E:conicalbasis}
  \Phi_\lambda(\varphi)=\sum_{s\in W} \varphi^*[m_sM,\lambda].
\end{equation}
\end{theorem}
\begin{proof}
If $\Phi$ is a conical distribution in $\mathcal D'_\lambda(\Xi_0)$
then we have shown that $\Phi$ is a multiple of $\Phi_\lambda$.
Conversely \eqref{E:eigendist17} shows that $\Phi_\lambda$ belongs to
$\mathcal D'_\lambda(\Xi_0)$, with $T=\sum_{s\in W} \delta_{m_sM}$;
clearly $T$ satisfies \eqref{E:q-inv} and is $M'$-invariant, so $\Phi_\lambda$ is conical.
\end{proof}

\section{The Case of Non-regular $\lambda$} 

Just as in the symmetric space case, the problem of characterizing the space of conical distributions in $\mathcal
D'_\lambda(\Xi_0)$ appears to be rather difficult, in general, when $\lambda\in\mathfrak a^*_c$ is
not regular.  One can, however, show that
the space of conical distributions corresponding to any non-regular
$\lambda$ is infinite-dimensional.  To see
this, let $K_\lambda=Z_K(\lambda)=\{k\in K\,|\,k\cdot \lambda=\lambda\}$
and let $K_\lambda'=\{k\in K\,|\,k\cdot \lambda\in \mathfrak a^*_c\}$.
For any $k\in K_\lambda'$, there exists an element $m'\in M'$ such
that $k\cdot \lambda=m'\cdot \lambda$.  Thus $(m')^{-1}k\in
K_\lambda$, and so we see that $K_\lambda'=M'K_\lambda=\cup_{s\in W}
m_s K_\lambda=\cup_{s\in W} K_{s\cdot \lambda} m_s$.

Let  $\Sigma_\lambda^+=\{\alpha\in\Sigma^+\,|\,B(\alpha,\lambda)=0\}$,
and as before let $W_\lambda$ be the subgroup of $W$ fixing $\lambda$.  Then
$W_\lambda$ is the subroup of $W$ generated by the reflections along
the root hyperplanes in $\Sigma_\lambda^+$, and
$M'\cap K_\lambda=\cup_{s\in W_\lambda} m_s M$.

The Lie algebra of $K_\lambda$ is $\mathfrak k_\lambda=\mathfrak
m+\sum_\alpha \mathfrak l_\alpha$, where the sum is taken over all
$\alpha$ in
$\Sigma^+_\lambda$.
  If $\lambda$ is
not regular, then $\Sigma_\lambda^+$ is nonempty, and therefore the orbit $K_\lambda/M$ is a
submanifold of $K/M$ of positive dimension. The set $M'K_\lambda/M$
is a disjoint union of $|W|/|W_\lambda|$ translates of $K_\lambda/M$, given by
$m_s\,K_\lambda/M$, where $s$ ranges over  a  set of coset representatives in $W/W_\lambda$.

Let $f$ be any continuous
function on the orbit $K_\lambda/M$, invariant under left translation by
elements of $m_s M$, for all $s\in W_\lambda$. Such $f$ can be
obtained by averaging any continuous function on the orbit by
$M$ and then further averaging by the $m_s$.  The vector space of such
$f$ is infinite-dimensional, since close to the identity coset $eM$,
the space of $M$-orbits in $K_\lambda/M$ is parametrized by the space
of $M$-orbits on a ball centered at $0$ in
$\sum_{\alpha\in\Sigma_\lambda^+} \mathfrak l_\alpha$.  

If $s\in W$, we can
extend $f$ in a well-defined way to the translated orbit $m_sK_\lambda/M$ by setting $f(m_s
kM)=f(kM)$, for all $k\in K_\lambda$.  In this way, $f$ becomes an $M'$-invariant function
defined on the union of the translated orbits $m_sK_\lambda/M$, for all $s\in W$.

Now let us define the distribution $T_f$ on $K/M$ by
\begin{equation}\label{E:tf-def}
  T_f(F)=
\sum_s\int_{K_\lambda/M}^{}f(m_s k_\lambda M)\,F(m_s k_\lambda M)\,d(k_\lambda)_M,\qquad (F\in\mathcal E(K/M))
\end{equation}
where the sum is taken over a set of representatives $s$ of
$W/W_\lambda$.  %(More precisely, over a set of representatives $s$
%of $W/W_\lambda$, and then, for each $s$, a fixed choice of $m_s\in
%M'$ representing $s$.)  
It is clear from the construction of $f$ that
$T_f$ is independent of the choice of the $m_s$ appearing on the right hand side above.
$T_f$ is then an $M'$-invariant distribution on $K/M$.  

Now, in accordance with Theorem \ref{T:geneigendist}, let us define
the distribution $\Phi_f$ in $\mathcal D'_\lambda(\Xi_0)$ by
\begin{equation}\label{E:conical1}
  \Phi_f(\varphi)=\int_{K/M}^{}\int_{\mathfrak a}^{}\widetilde\varphi(kM,H)\,e^{i\lambda(H)}\,dH\,dT_f(kM)
\end{equation}

Since $T_f$ is $M'$-invariant, so is $\Phi_f$.  To show that $\Phi_f$
is $\mathfrak q$-invariant, we use the expression \eqref{E:tf-def}
defining $T_f$:
$$
  \Phi_f(\varphi)=\sum_s\int_{K_\lambda/M}^{}\int_{\mathfrak
  s}^{}\widetilde\varphi(m_sk_\lambda M, H)\,e^{i\lambda(H)}\,dH\,d(k_\lambda)_M
$$
Let $X\in\mathfrak q$.  Then by \eqref{E:gpaction1} we have 
\begin{align*}
  \Phi_f(\varphi^{\tau(X)})=\sum_s \int_{K_\lambda/M}^{}\int_{\mathfrak
  s}^{}\widetilde\varphi(m_sk_\lambda M, H)\,e^{i\lambda(H)}\,e^{iB(m_s k_\lambda\cdot \lambda,X)}\,dH\,d(k_\lambda)_M
\end{align*}
But $m_s k_\lambda\cdot \lambda\in\mathfrak a^*_c$, and thus $B(m_s
k_\lambda\cdot \lambda,X)=0$, which shows that the right hand side
above equals $\Phi_f(\varphi)$.  

As remarked above, the space
of all continuous functions $f$ on $K_\lambda/M$ invariant under the
left action of $m_sM$, for all $s\in W_\lambda$, is
infinite-dimensional.  Hence the space of conical distributions in
$\mathcal D'_\lambda(\Xi_0)$ has infinite dimension.

In the case when the symmetric space $X=G/K$ has rank one; i.e., when
$\dim \mathfrak a=1$, it is possible to obtain a complete classification of the space of all
conical distributions in $\mathcal D'_0(\Xi_0)$.  In this case,
$\Sigma^+$ has one or two elements; let $\alpha$ be the indivisible
element.  Choose $H\in\mathfrak a$ such that $\alpha(H)=1$, and
identify  $\mathbb R$ with $\mathfrak a$ by $t\mapsto tH$.

Since we are assuming that $\lambda=0$, then $W_\lambda=W=\{\pm 1\}$, and so the space $H$ of
$W_\lambda$-harmonic 
polynomials on $\mathfrak a$ has basis $\{1,\,t\}$.  Suppose
that $\Phi\in\mathcal D'_0(\Xi_0)$ is a conical distribution.  Then from
 \eqref{E:eigendist16}, there exist uniquely determined $M'$-invariant
 distributions $T_0$ and $T_1$ on $K/M$ such that
\begin{equation}\label{E:rankone1}
  \Phi(\varphi)=\int_{K/M}^{}\int_{-\infty}^{\infty}\widetilde\varphi(kM,tH)\,dt\,dT_0(kM)+
\int_{K/M}^{}\int_{-\infty}^{\infty}\widetilde\varphi(kM,tH)\,t\,dt\,dT_1(kM)
\end{equation}
Since $\Phi$ is also invariant under left translation by any
$X\in\mathfrak q$, we have
\begin{align*}
  \Phi(\varphi)&=\int_{K/M}^{}\int_{-\infty}^{\infty}\widetilde\varphi(kM,tH+(k^{-1}\cdot X)_\mathfrak a)\,dt\,dT_0(kM)\\
&\qquad\qquad\qquad+
\int_{K/M}^{}\int_{-\infty}^{\infty}\widetilde\varphi(kM,tH+(k^{-1}\cdot
X)_\mathfrak a)\,t\,dt\,dT_1(kM)\\
&=\int_{K/M}^{}\int_{-\infty}^{\infty}\widetilde\varphi(kM,tH)\,dt\,dT_0(kM)+
\int_{K/M}^{}\int_{-\infty}^{\infty}\widetilde\varphi(kM,tH)\,t\,dt\,dT_1(kM)\\
&\qquad\qquad\qquad-\int_{K/M}^{}\int_{-\infty}^{\infty}\widetilde\varphi(kM,tH)\,B(k\cdot A_\alpha,X)dt\,dT_1(kM)\\
&=\Phi(\varphi)-\int_{K/M}^{}\int_{-\infty}^{\infty}\widetilde\varphi(kM,tH)\,dt\,\,B(k\cdot A_\alpha,X)\,dT_1(kM)
\end{align*}
Hence
\begin{equation}\label{E:q-trans1}
  \int_{K/M}^{}\int_{-\infty}^{\infty}\widetilde\varphi(kM,tH)\,dt\,\,B(k\cdot A_\alpha,X)\,dT_1(kM)=0
\end{equation}
for all $\varphi\in\mathcal D(\Xi_0)$.

If $T_0$ and $T_1$ are $M'$-invariant distributions on $K/M$ it is
clear that the condition \eqref{E:q-trans1} is also sufficient for the
distribution $\Phi$ in \eqref{E:rankone1} to be conical in $\mathcal
D'_0(\Xi_0)$.  In particular, $T_0$ can be arbitrary.

Now it is easy to see that the map
$\varphi\mapsto\int_{-\infty}^{\infty}\widetilde\varphi(kM,tH)\,dt$
maps $\mathcal D(\Xi_0)$ onto the vector space $\mathcal E_{M'}(K/M)$ of $C^\infty$ functions
$F$ on $K/M$ satisfying $F(kM)=F(km^*M)$ for all $k\in K$, where
$m*$ is any element in $M'\setminus M$.   Thus \eqref{E:q-trans1}
implies that $\Phi$ is conical if and only if the $M'$-invariant
distribution $T_1$ satisfies the condition
\begin{equation}\label{E:q-trans2}
  \int_{K/M}^{}\!F(kM)\,B(k\cdot A_\alpha,X)\,dT_1(kM)=0
\end{equation}
for any $X\in\mathfrak q$ and all $F\in\mathcal E_{M'}(K/M)$.  As we
shall show below, it turns out that {\it all} $M'$-invariant
distributions on $K/M$ satisfy the condition above.

Since $\dim \mathfrak a=1$, the set
$\Sigma^+$ consists of $\alpha$, and possibly $2\alpha$, with
multiplicities $m_\alpha$ and $m_{2\alpha}$, respectively.   Let $H_1$ be the
unit vector in $\mathfrak a$ such that $\alpha(H_1)>0$, and let $o$ denote the
identity coset $\{M\}$ in $K/M$.  Then we can endow $K/M$ with the
$K$-invariant Riemannian structure induced from the $\text{Ad}\,M$-invariant inner product on
$\mathfrak l\cong T_o(K/M)$ given by 
$$
\langle
T_\alpha+T_{2\alpha},T_\alpha'+T_{2\alpha}'\rangle=-\alpha(H_1)^2\,
B(T_\alpha,T_\alpha')-4\alpha(H_1)^2\,B(T_{2\alpha},T_{2\alpha}')
$$
for $T_\alpha,T_\alpha'\in\mathfrak l_\alpha$ and
$T_{2\alpha},T_{2\alpha}'\in\mathfrak l_{2\alpha}$.  One can easily show that the mapping
$kM\mapsto k\cdot H_1$  is  an isometry
from $K/M$ onto the unit sphere $S$ in $\mathfrak p$.  Whenever it is convenient, we will
identify $K/M$ with $S$ in this manner.

\begin{lemma}\label{T:conjugation}
Assume that $\dim \mathfrak a=1$.  Fix $m^*\in M'\setminus M$.  Then for every $kM\in K/M$, there exists an
$m\in M$ such that $m^*k(m^*)^{-1}M=mkM$.
\end{lemma}
\begin{proof}
It is easy to see that the map $kM\mapsto m^*k(m^*)^{-1}M$ is a
well-defined isometry of $K/M$.  By Theorem 13.2 in \cite{Nom}, the map $T
\mapsto (\exp T)\,M$ maps $\mathfrak l$ onto $K/M$, and clearly
$m^*(\exp T)(m^*)^{-1}M=\exp (\text{Ad}(m^*)\,T)M$.  Thus it suffices
to prove that for each $T\in\mathfrak l$, there exists $m\in M$ such
that $\text{Ad}(m^*)\,T=\text{Ad}(m)\,T$.

This assertion can be proved by considering the possible
cases for $m_\alpha$ and $m_{2\alpha}$.  For convenience, let us now
provide $\mathfrak l$ with the inner product given by
$-B$, which we note that $\text{Ad}\,(m^*)$ leaves invariant.  Suppose first that $m_{2\alpha}>1$.  Write
$T\in\mathfrak l$ as $T=T_\alpha+T_{2\alpha}$, with
$T_\alpha\in\mathfrak l_\alpha,\,T_{2\alpha}\in\mathfrak l_{2\alpha}$.
For any $r,s\geq 0$, $\text{Ad}\, M$ is transitive on the product of
spheres $\{T'+T''\in\mathfrak l_\alpha+\mathfrak
l_{2\alpha}\,|\,\|T'\|=r,\,\|T''\|=s\}$ (\cite{Ko}).  Since
$\text{Ad}\,(m^*)$ 
is an isometry on $\mathfrak l_\alpha$ and on $\mathfrak l_{2\alpha}$, there
exists $m\in M$ such that
$\text{Ad}(m)T_\alpha=\text{Ad}(m^*)T_\alpha$ and
$\text{Ad}(m)T_\alpha=\text{Ad}(m^*)T_{2\alpha}$.

Suppose next that $m_{2\alpha}=0$ and $m_\alpha>1$.  Then $\mathfrak
l=\mathfrak l_\alpha$, and since $\text{Ad} M$ is transitive on spheres
in $\mathfrak l_\alpha$, our assertion easily holds in this case.

The remaining cases are $m_{2\alpha}=1$ (so $m_\alpha>1$) and
$m_\alpha=1$ (so $m_{2\alpha}=0$). Suppose that $m_{2\alpha}=1$.  We
claim that $\text{Ad}\,m^*$ is the identity map on $\mathfrak
l_{2\alpha}$. For this, we recall that $\mathfrak g$ has
 decomposition $\mathfrak
g=\mathfrak g_{-2\alpha}+\mathfrak g_{-\alpha}+\mathfrak m+\mathfrak
a+\mathfrak g_\alpha+\mathfrak g_{2\alpha}$.  Choose any nonzero
elements $X_\alpha\in\mathfrak g_\alpha$ and $X_{2\alpha}\in\mathfrak
g_{2\alpha}$.  Then $X_\alpha,\,\theta(X_\alpha),\,X_{2\alpha}$, and
$\theta(X_{2\alpha})$ generate a Lie subalgebra $\mathfrak g^*$ of
$\mathfrak g$ isomorphic to $\text{su}(2,1)$. Let
$G^*$ be the analytic subgroup of $G$ with this algebra.  Then $G^*$
has Iwasawa decomposition $G^*=K^*A N^*$, where $K^*=G^*\cap
K,\,N^*=G^*\cap N$.  If $M^*$ and $(M')^*$ denote the centralizer and
normalizer of $\mathfrak a$ in $K^*$, we also have $M^*=G^*\cap M$ and
$(M')^*=G^*\cap M'$.  Choose any element $m^*_1\in (M')^*\setminus
M^*$.  Then $m_1^*\in M'\setminus M$ so there exists an $m_1\in M$
such that $m_1^*=m^*m_1$.  Now from \cite{DS}, Chapter IX, \S3,
$\text{Ad}\,((m')^*)\,X_{2\alpha}=\theta(X_{2\alpha}),\;\text{Ad}\,((m')^*)\,\theta(X_{2\alpha})=X_{2\alpha}$,
 and thus
$\text{Ad}\,(m')^*$ fixes $X_{2\alpha}+\theta(X_{2\alpha})$.  But this
latter vector spans $\mathfrak l_{2\alpha}$.  Since $\text{Ad}\,M$ is
the identity map on $\mathfrak l_{\pm 2\alpha}$ (\cite{GASS}, Chapter III, Lemma 3.8) it follows that
$\text{Ad}\,m^*=\text{Ad}\,(m_1^*\,m_1^{-1})$ is the identity map on
$\mathfrak l_{\pm 2\alpha}$ as well.  Now since $\text{Ad}\,M$ is
transitive on spheres in $\mathfrak l_\alpha$, we conclude that for
any $T\in\mathfrak l_\alpha$ and $T'\in\mathfrak l_{2\alpha}$, there
exists an $m\in M$ such that $\text{Ad}\,(m^*)\,T=\text{Ad}\,(m)\,T$ and
$\text{Ad}\,(m^*)\,T'=\text{Ad}\,(m)\,T'=T'$.

Finally, suppose that $m_\alpha=1$. Choose any nonzero $X_\alpha\in
\mathfrak g_\alpha$.  Then $X_\alpha$ and $\theta(X_\alpha)$ generate a
subalgebra $\mathfrak g^*$ of $\mathfrak g$ isomorphic to
$\text{su}\,(1,1)$.  Let $G^*$ be the analytic subgroup of $G$ with
Lie algebra $\mathfrak g^*$, and let $(m')^*\in G^*\cap (M'\setminus
M)$.  There exists an $m_1\in M$ such that $(m')^*=m^*\,m_1$.  Now an
 easy matrix computation on $\text{SU}\,(1,1)$ shows that
 $\text{Ad}\,(m')^*\,X_\alpha=\theta(X_\alpha)$ and
 $\text{Ad}\,(m')^*\,\theta(X_\alpha)
=X_\alpha$, from which we again
 conclude that $\text{Ad}\,(m')^*$, and hence $\text{Ad}\,(m^*)$ is
 the identity map on $\mathfrak l_\alpha$. On the other hand, by \cite{GASS},
 Chapter III, Lemma 3.8, $\text{Ad}\,(M)$ is
 also the identity map on $\mathfrak l_\alpha$. Thus
 $\text{Ad}\,(m^*)\,T=\text{Ad}\,(m)\,T=T$ for all $m\in M$ and
 $T\in\mathfrak l_\alpha$.

This covers all the cases and finishes the proof of the lemma.
\end{proof}

We are now in a position to classify the conical distributions in
$\mathcal D'_0(\Xi_0)$ when $\dim \mathfrak a=1$.

\begin{theorem}\label{T:rankoneclass}
Assume that $\dim \mathfrak a=1$.  Then the conical distributions in
$\mathcal D'(\Xi_0)$ are precisely those distributions $\Phi$ given by
\begin{equation}\label{E:conicalclass}
  \Phi(\varphi)=\int_{K/M}^{}\int_{-\infty}^{\infty}\widetilde\varphi(kM,tH)\,dt\,dT_0(kM)+
\int_{K/M}^{}\int_{-\infty}^{\infty}\widetilde\varphi(kM,tH)\,t\,dt\,dT_1(kM)
\end{equation}
where $T_0$ and $T_1$ are $M'$-invariant distributions on $K/M$.
\end{theorem}
\begin{proof}
As remarked in \eqref{E:rankone1}, any conical distribution $\Phi$ in
$\mathcal D'_0(\Xi_0)$ must be of the form \eqref{E:conicalclass},
with $T_0$ and $T_1$ $M'$-invariant.

Conversely, suppose that $\Phi\in \mathcal D'(\Xi_0)$ is defined by
\eqref{E:conicalclass} with $T_0$ and $T_1$ $M'$-invariant. By Theorem
\ref{T:geneigendist}, $\Phi$ belongs to $\mathcal D'_0(\Xi_0)$, and it
is clear that $\Phi$ is $M'$-invariant.  To prove that $\Phi$ is
conical, it is sufficient to verify that $T_1$ satisfies
\eqref{E:q-trans2} for all $F\in\mathcal E(K/M)$ such that
$F(kM)=F(km^*M)$.  

To this end,   let us put
$F^\#(kM)=\int_{M'}^{}F(m'kM)\,dm'$ for any function $F\in\mathcal E(K/M)$, where $dm'$ is the normalized Haar measure
on the compact group $M'$.   Note that since $T_1$ is $M'$-invariant,
$T_1(F)=T_1(F^\#)$ for all $F\in\mathcal E(K/M)$.

Lemma \ref{T:conjugation} shows that for any $kM\in K/M$, there exists
$m_1\in M$ such that $m^*k\cdot
H_1=-m^*k(m^*)^{-1}\cdot H_1=-m_1k\cdot H_1$.  If $a:\omega\mapsto-\omega$
denotes the antipodal map on the sphere $S$, then $a(k\cdot H_1)=k(m^*)^{-1}\cdot
H_1$, so $a$ corresponds to the isometry $kM\mapsto k(m^*)^{-1}M$ of
$K/M$.

 Noting that $F^a(kM)=F(k(m^*)^{-1}M)$, we
see from the definition of $F^\#$ that
that $(F^\#)^a=(F^a)^\#$. On the other hand, for any $kM\in K/M$, we put $k\cdot
H_1=\omega$. Applying Lemma \ref{T:conjugation}, we have
\begin{align*}
  (F^a)^\#(\omega)&=\int_{M'}^{}F(m'k(m^*)^{-1}M)\,dm'\\
&=\int_{M'}^{}F(m'm^*k(m^*)^{-1}M)\,dm'\\
&=\int_{M'}^{} F(m'm_1kM)\,dm'\qquad\qquad (\text{for some}\;m_1\in M)\\
&=F^\#(\omega).
\end{align*}
In particular, if $F\in\mathcal E(K/M)$ corresponds to an odd function on $S$, we have $F^\#=0$.

Now suppose that $F\in\mathcal E(K/M)$ satisfies $F(km^*M)=F(kM)$ for
all $k\in K$.  Then $F$ corresponds to an even function on $S$, and
for each fixed $X\in\mathfrak q$, the function $G(kM)=F(kM)\,B(k\cdot A_\alpha, X)$ 
is an odd function on $S$.  It follows that $G^\#=0$ and thus
\begin{align*}
  \int_{K/M}^{}\!F(kM)\,B(k\cdot A_\alpha,X)\,dT_1(kM)&=T_1(G)\\
&=T_1(G^\#)\\
&=0.
\end{align*}

Thus \eqref{E:q-trans2} holds for all such $F$, and we conclude that
the distribution $\Phi$ in \eqref{E:rankone1} is conical if and only
if $T_0$ and $T_1$ are $M'$-invariant.
\end{proof}

It is curious that the $M'$-invariance of any distribution in $\mathcal D'_0(\Xi_0)$
guarantees its $\mathfrak q$-invariance.

\section{Acknowledgements}
The author would like to express his gratitude to Prof. S. Helgason for his
valuable assistance in the preparation of this paper and in particular
in the proof of Lemma \ref{T:conjugation}.


\begin{thebibliography}{9}
\bibitem{Duality} S. Helgason, \emph{A duality in integral geometry
    with applications to group representations}, Adv. Math. {\bf 5}
  (1970), pp. 1-154.
\bibitem{Duality3} S. Helgason, \emph{A duality in integral geometry
    with applications to group representations, III.  Tangent space analysis}, Adv. Math. {\bf 30}
  (1980), pp. 297-323.
\bibitem{DS} S. Helgason, \emph{Differential Geometry, Lie Groups, and
    Symmetric Spaces}, American Mathematical Society,  Providence RI
  (2001).
\bibitem{GGA} S. Helgason, \emph{Groups and Geometric Analysis}, American Mathematical Society,  Providence RI (2002).
\bibitem{GASS} S. Helgason, \emph{Geometric Analysis on Symmetric Spaces}, 2nd ed. 
American Mathematical Society, Providence RI (2008).
\bibitem{Hu} M.C. Hu, \emph{Determination of the conical distributions
    for rank one symmetric spaces}, Thesis, MIT (1973).
\bibitem{Ko} B. Kostant, \emph{On the existence and irreducibility of
    certain series of representations}, in ``Lie Groups and Their
  Representations'' (I.M. Gelfand, ed.), pp. 231-239, Halstead, New
  York, 1975.
\bibitem{Nom} K. Nomizu, \emph{Invariant affine connections on homogeneous spaces}, Amer. J. Math.
{\bf 76} (1954) pp. 33-65.
\bibitem{Orloff1} J. Orloff, \emph{Invariant Radon transforms on a
    symmetric space}, Contemp. Math. {\bf 113} (1990), pp. 233-242.
\bibitem{Orloff2} J. Orloff, \emph{Invariant Radon transforms on a
    symmetric space}, Trans. Amer. Math. Soc. {\bf 318} (1990), pp.581-600.
\end{thebibliography}
\end{document}